\documentclass[12pt]{amsart}
\usepackage{amsmath}
\usepackage{amsthm}
\usepackage{amssymb}

\numberwithin{equation}{section}
\newtheorem{Theorem}{Theorem}
\newtheorem{Lemma}[Theorem]{Lemma}

\newtheorem{Proposition}[Theorem]{Proposition}

\newtheorem{Corollary}[Theorem]{Corollary}

\newcommand{\be}{\begin{equation}}
\newcommand{\ee}{\end{equation}}
\newcommand{\phii}{\varphi}
\newcommand{\e}{\varepsilon}

\begin{document}

\title[Characterization of potential smoothness]
{Characterization of potential smoothness and Riesz basis property
of the Hill-Scr\"odinger operator in terms of periodic, antiperiodic
and Neumann spectra}

\author{Ahmet Batal}
\address{Sabanci University, Orhanli,
34956 Tuzla, Istanbul, Turkey}

\email{ahmetbatal@sabanciuniv.edu}

\begin{abstract}
The Hill operators $Ly=-y''+v(x)y$, considered with complex valued
$\pi$-periodic potentials $v$ and subject to periodic, antiperiodic
or Neumann boundary conditions have discrete spectra. For
sufficiently large $n,$ close to $n^2$ there are two periodic (if
$n$ is even) or antiperiodic (if $n$ is odd) eigenvalues
$\lambda_n^-$, $\lambda_n^+$ and one Neumann eigenvalue $\nu_n$. We
study the geometry of ``the spectral triangle'' with vertices
($\lambda_n^+$,$\lambda_n^-$,$\nu_n$), and show that the rate of
decay of triangle size characterizes the potential smoothness.
Moreover, it is proved, for  $v\in L^p ([0,\pi]), \; p>1,$ that the
set of periodic (antiperiodic) root functions contains a Riesz basis
if and only if for even $n$ (respectively, odd $n$)  $ \;
\sup_{\lambda_n^+\neq
\lambda_n^-}\{|\lambda_n^+-\nu_n|/|\lambda_n^+-\lambda_n^-| \} <
\infty. $
\vspace{1mm}\vspace{1mm} 

\end{abstract}

\maketitle

\section{Introduction}
The Schr\"odinger operator
 \be
 \label{1}
  L y= -y''+v(x)y, \quad \quad \quad x \in \mathbb{R},
 \ee
considered with a real-valued $\pi$-periodic potential $v \in
L^2([0,\pi]),$ is self-adjoint and its spectrum has a band-gap
structure, i.e., it consists of intervals $[\lambda^+_{n-1},
\lambda^-_{n}]$ separated by spectral gaps (instability zones)
$(\lambda^-_{n}, \lambda^+_{n})$, $n \in \mathbb{N}$. The Floquet
theory (e.g., see \cite{E}) shows that the endpoints $\lambda^-_{n}$,
$\lambda^+_{n}$ of these gaps are eigenvalues of the same
differential operator (\ref{1}) but considered on the interval
$[0,\pi]$ with periodic boundary conditions for even $n$ and
antiperiodic boundary conditions for odd $n$.

Hochstadt \cite{Ho1, Ho2} discovered that there is a close relation
between the rate of decay of the {\em spectral gap}
$\gamma_n=\lambda^+_n-\lambda^-_n$ and the smoothness of the
potential $v$. He proved that every finite zone potential is a
$C^\infty $-function, and moreover, \emph{if v is infinitely
differentiable then $\gamma_n$ decays faster than any power of
$1/n.$} Later several authors \cite{LP}-\cite{MT} studied this
phenomenon and showed that \emph{if $\gamma_n$ decays faster than any
power of $1/n,$ then v is infinitely differentiable}. Moreover,
Trubowitz \cite{Tr} proved that \emph{$v$ is analytic if and only if
$\gamma_n$ decays exponentially fast}.

If $v$ is a complex-valued function then the operator (\ref{1}) is
not self-adjoint and we cannot talk about spectral gaps. But
$\lambda^\pm_n $ are well defined for sufficiently large $n$ as
eigenvalues of (\ref{1}) considered on the interval $[0,\pi]$ with
periodic or antiperiodic boundary conditions, so we set again
$\gamma_n = \lambda_n^+ - \lambda_n^-$ and call it spectral gap.
Again the potential smoothness determines the decay rate of
$\gamma_n,$ but in general the opposite is not true. The decay rate
of $\gamma_n$ has no control on the smoothness of a complex valued
potential $v$ by itself as the Gasymov paper \cite{Gas} shows.

Tkachenko \cite{Tk92}--\cite{Tk94} discovered that the smoothness of
complex potentials could be controlled if one consider, together
with the spectral gap $\gamma_n,$ the deviation
$\delta^{Dir}_n=\lambda^+_n-\mu_n$, where $\mu_n$ is the closest to
$n^2$ Dirichlet eigenvalue of $L.$ He characterized in these terms
the $C^{\infty}$-smoothness and analyticity of complex valued
potentials $v.$  Moreover, Sansuc and Tkachenko \cite{ST} showed
that $v$ is in the Sobolev space $H^a$, $a\in \mathbb{N}$ if and
only if $\gamma_n $ and $\delta^{Dir}_n$ are in the weighted
sequence space $\ell_a^2=\ell^2((1+n^2)^{a/2})$.

The above results have been obtain by using Inverse Spectral Theory.
Kappeler and Mityagin \cite{KaMi01} suggested another approach based
on Fourier Analysis. To formulate their results, let us recall that
the smoothness of functions could be characterized by weights $\Omega
= (\Omega (k)),$  and the corresponding weighted Sobolev spaces are
defined by
$$
H(\Omega) = \{v(x) = \sum_{k \in \mathbb{Z}} v_k e^{2ikx}, \quad
\sum_{k \in \mathbb{Z}} |v_k|^2 (\Omega(k))^2 < \infty \}.
$$
A weight $\Omega $ is called sub-multiplicative, if $ \Omega (-k) =
\Omega (k) $ and $ \Omega (k+m) \leq \Omega (k)\Omega (m) $ for $ k,m
\geq 0. $ In these terms the main result in \cite{KaMi01} says that
if $\Omega $ is a sub-multiplicative weight, then
 \be
  \label{i-ra}
 (A) \quad v\in H(\Omega) \quad \Longrightarrow \quad  (B) \quad
   (\gamma_n) ,  \left (\delta^{Dir}_n \right ) \in
  \ell^2(\Omega).
 \ee
Djakov and Mityagin \cite{DM3,DM5, DM15} proved the inverse
implication $(B) \Rightarrow (A)$  under some additional mild
restrictions on the weight $\Omega. $  Similar results were obtained
for 1D Dirac operators (see \cite{DM7,DM6,DM15}).

The analysis in \cite{KaMi01, DM3, DM5, DM15} is carried out under
the assumption $v\in L^2([0,\pi]).$  Using quasi-derivatives, Djakov
and Mityagin \cite{DM16} developed a Fourier method for studying the
spectra of $L_{Per^\pm}$ and $L_{Dir}$ in the case of periodic
singular potentials and  extended the above results. They proved that
if  $v \in H^{-1}_{per} (\mathbb{R})$ and $\Omega$ is a weight of the
form $\Omega(m)=\omega(m)/|m|$ for  $m \neq 0,$ with $\omega$ being a
sub-multiplicative weight, then  $(A) \Rightarrow (B),  $ and
conversely, if in addition $(\log \omega (n))/n$ decreases to zero,
then $(B) \Rightarrow (A)$ (see Theorem 37 in \cite{DMH}).

A crucial step in proving the implications $(A) \Rightarrow (B)$ and
$(B) \Rightarrow (A)$ is the following statement (which comes from
Lyapunov-Schmidt projection method, e.g., see Lemma 21 in
\cite{DM15}): {\em For large enough $n,$ the number $\lambda = n^2 +
z, \; |z| <n/4 $ is a periodic or antiperiodic eigenvalue if and only
if $z$ is an eigenvalue of the matrix} $
\begin{pmatrix} \alpha_n(z) & \beta^+_n(z) \\ \beta^-_n(z) &
\alpha_n(z)\end{pmatrix}.$ The entrees of this matrix
$\alpha_n(z)=\alpha_n(z;v)$ and $\beta^{\pm}_n(z)=\beta^{\pm}_n(z;v)$
are given by explicit expressions in terms of the Fourier
coefficients of the potential $v$ and depend analytically on $z$ and
$v.$  The functionals $\beta^\pm_n $ give lower and  upper bounds for
the gaps and deviations (e.g., see Theorem 29 in  \cite{DMH}): {If
$v \in H^{-1}_{per} (\mathbb{R})$ then, for sufficiently large} $n,$
\begin{equation}
\label{A1} \frac{1}{72}(|\beta_n^+(z_n^*)|+|\beta_n^-(z_n^*)|) \leq
|\gamma_n|+|\delta^{Dir}_n| \leq 58
 (|\beta_n^+(z_n^*)|+|\beta_n^-(z_n^*)|),
\end{equation}
where $z_n^*=  \frac{1}{2}(\lambda^+_n + \lambda^-_n) - n^2.$ Thus,
the implications $(A) \Rightarrow (B)$ and $(B) \Rightarrow (A)$ are
equivalent, respectively, to
\begin{equation}
\label{A2} (\tilde{A}):\quad v\in H(\Omega) \quad \Longrightarrow
\quad (\tilde{B}): \quad (|\beta_n^+(z_n^*)|+|\beta_n^-(z_n^*)|)\in
\ell^2 (\Omega),
\end{equation}
and $(\tilde{B}) \Rightarrow (\tilde{A}).$ In this way the problem of
analyzing the relationship between potential smoothness and decay
rate of the sequence $(|\gamma_n| + |\delta_n^{Dir}|)$ is reduced to
analysis of the functionals $\beta^{\pm}_n (z).$

The asymptotic behavior  $\beta^{\pm}_n (z)$ (or $\gamma_n$ and
$\delta^{Dir}_n$) plays also a crucial role in studying the Riesz
basis property of the system of root functions of the operators
$L_{Per^\pm}.$ In \cite[Section 5.2]{DM15}, it is shown (for
potentials $v \in L^2 ([0,\pi])$) that if the ratio $\beta^+_n
(z_n^*)/\beta^-_n (z_n^*) $ is not separated from $0$ or $\infty$
then the system of root functions of $L_{Per^\pm}$ does not contain a
Riesz basis (see Theorem 71 and its proof therein). Theorem 1 in
\cite{DM25} (or Theorem 2 in \cite{DM25a}) gives, for wide classes of
$L^2$-potentials,  a criterion for Riesz basis property in the same
terms. In its most general form, for singular potentials, this
criterion reads as follows (see Theorem 19 in \cite{DM28}):

{\em Criterion 1. Suppose $v\in H^{-1}_{per}(\mathbb{R});$ then the
set of root functions of $L_{Per^\pm}(v)$ contains Riesz bases if
and only if
\begin{equation}
\label{DM1} 0 < \inf_{\gamma_n\neq 0} |\beta^-_n (z_n^*)|/|\beta^+_n
(z_n^*)|, \quad \sup_{\gamma_n\neq 0} |\beta^-_n (z_n^*)|/|\beta^+_n
(z_n^*)|< \infty,
\end{equation}
where $n$ is even (respectively odd) in the case of periodic
(antiperiodic) boundary conditions.}

In \cite{GT11}  Gesztesy and Tkachenko obtained the following result.

{\em Criterion 2.   If $v \in L^2 ([0,\pi]),$ then there is a Riesz
basis consisting of root functions of the operator $L_{Per^\pm}(v)$
if and only if
\begin{equation}
\label{GT1} \sup_{\gamma_n\neq 0}|\delta^{Dir}_n|/|\gamma_n|<\infty,
\end{equation}
where $n$ is even (respectively odd) in the case of periodic
(antiperiodic) boundary conditions.}

They also noted that a similar criterion holds if (\ref{GT1}) is
replaced by
\begin{equation}
\label{GT2} \sup_{\gamma_n\neq 0}|\delta^{Neu}_n|/|\gamma_n|<\infty,
\end{equation}
where $\delta^{Neu}_n=\lambda^+_n - \nu_n $ and $\nu_n $ is the $n$th
Neumann eigenvalue.

Djakov and Mityagin \cite[Theorem 24]{DM28} proved, for potentials
$v\in H^{-1}_{per}(\mathbb{R}),$ that the conditions ({\ref{DM1}) and
(\ref{GT1}) are equivalent, so (\ref{GT1}) gives necessary and
sufficient conditions for Riesz basis property for singular
potentials as well.

However, if $v \not \in L^2 ([0,\pi])$ (say, for potentials $v \in
L^1 ([0,\pi])$ or  $v\in H^{-1}_{per}(\mathbb{R})$) \emph{it is not
known whether (\ref{GT2}) is equivalent to Riesz basis property.}

In this paper, we show for potentials $v\in L^p([0,\pi]), $ $ p>1,$
that the Neumann version of Criterion 2 holds, and the potential
smoothness could be characterized by the decay rate of $ |\gamma_n|
+ |\delta_n^{Neu}|.$ More precisely, the following two theorems
hold.

\begin{Theorem}
\label{ThmNeu}  Suppose $v\in L^p([0,\pi]), \; p>1$ and $\Omega$  is
a weight of the form $\Omega(m)=\omega(m)/m$ for  $m \neq 0,$ where
$\omega$ is a sub-multiplicative weight. Then
\begin{equation}
\label{B10} v \in H(\Omega) \Longrightarrow (|\gamma_n|),
(|\delta^{Neu}_n|) \in \ell^2 (\Omega);
\end{equation}
conversely, if in addition $(\log \omega (n))/n$ decreases to zero,
then
\begin{equation}
\label{B20}  (|\gamma_n|), (|\delta^{Neu}_n|) \in \ell^2
(\Omega)\Longrightarrow v \in H(\Omega).
\end{equation}
If $ \lim \frac{\log \omega (n)}{n} >0,$ (i.e. $\omega $ is of
exponential type), then
\begin{equation}
\label{B30} (\gamma_n), (\delta^{Neu}_n) \in \ell^2 (\Omega) \;
\Rightarrow \; \exists  \varepsilon >0: \; v \in H(e^{\varepsilon
|n|} ).
\end{equation}
\end{Theorem}

 \begin{Theorem}
 \label{Theorem}
  If $v \in L^p ([0,\pi]), \; p>1,$ then there is a Riesz
basis consisting of root functions of the operator $L_{Per^\pm}(v)$
if and only if
\begin{equation}
\label{GT10} \sup_{\gamma_n\neq 0}|\delta^{Neu}_n|/|\gamma_n|<\infty,
\end{equation}
where $n$ is respectively even (odd) for periodic (antiperiodic)
boundary conditions.
 \end{Theorem}

 We do not prove
Theorem~\ref{ThmNeu} and Theorem~\ref{Theorem} directly, but show
that they are valid by reducing their proofs to Theorem~37 in
\cite{DMH} and Theorem~19 in \cite{DM28}, respectively. For this end
we prove the following.

 \begin{Theorem}
 \label{Asymp}
If $v \in L^p ([0,\pi]), \; p>1,$ then, for sufficiently large $n,$
\begin{equation}
\label{A10} \frac{1}{80} (|\beta_n^+(z_n^*)|+|\beta_n^-(z_n^*)|)
\leq |\gamma_n|+|\delta^{Neu}_n| \leq 19
 (|\beta_n^+(z_n^*)|+|\beta_n^-(z_n^*)|).
\end{equation}
\end{Theorem}

Next we show that Theorem \ref{Asymp} implies Theorem~\ref{ThmNeu}
and Theorem~\ref{Theorem}}. Fix $v \in L^1 ([0,\pi]).$ We may assume
without loss of generality that $ V(0):=\int_0^\pi v(t) dt = 0 $
(otherwise one may consider $\tilde{v}(t) = v(t) -\frac{1}{\pi}
V(0)$ since a constant shift of the potential results in a shift the
spectra but root functions, spectral gaps and deviations  do not
change). Then the function $Q(x)= \int_0^x v(t) dt, $ extended
periodically on $\mathbb{R},$ has the property that $v= Q^\prime $
so we may think that $v\in H^{-1}_{per}(\mathbb{R})$  (see
\cite{DM16}) for details). By Theorem~29 in \cite{DMH} and
Theorem~\ref{Asymp}, (\ref{A1}) and (\ref{A10}) hold simultaneously,
so the sequences $(|\gamma_n|+|\delta^{Dir}_n|)$ are
$(|\gamma_n|+|\delta^{Neu}_n|)$ are asymptotically equivalent.
Therefore, every claim in Theorem~\ref{ThmNeu} follows from the
corresponding assertion in
 \cite[Theorem 37]{DMH}.

On the other hand the asymptotic equivalence of
$|\gamma_n|+|\delta^{Dir}_n|$  and $|\gamma_n|+|\delta^{Neu}_n|$
imply the asymptotic equivalence of
$\frac{|\delta_n^{Neu}|}{|\gamma_n|}$ and
$\frac{|\delta^{Dir}|}{|\gamma_n|},$ so (\ref{GT1}) and (\ref{GT10})
hold simultaneously if $v \in L^p ([0,\pi]).$ By Theorem 24 in
\cite{DM28}, (\ref{GT1}) gives necessary and sufficient conditions
for the Riesz basis property if $v\in H^{-1}_{per} (\mathbb{R}).$
Hence, Theorem~\ref{Theorem} is proved.

\section{Preliminary Results}
\vspace{3mm}

We consider the Hill--Schr\"odinger operator
\begin{equation}
\label{11.1} Ly = -y^{\prime \prime} + v(x) y
\end{equation}
with a (complex--valued) potential $v\in L^1([0,\pi])$, subject to
the following boundary conditions ($bc$):

($a$) {\em periodic} $(Per^+): \quad y(0) = y (\pi), \;\; y^\prime
(0) = y^\prime (\pi) $;

($b$) {\em anti--periodic}  $(Per^-): \quad  y(0) = - y (\pi), \;\;
y^\prime (0) = - y^\prime (\pi) $;

($c$) {\em Dirichlet} $ (Dir): \quad y(0) = 0, \;\; y(\pi) = 0; $

($d$) {\em Neumann} $ (Neu): \quad y'(0) = 0, \;\; y'(\pi) = 0. $

The closed operator generated by $L$  in the domain $$ Dom(L_{bc}) =
\{f' \text{ is absolutely continuous} : \; \; f \in (bc), \; Lf\in
L^2([0,\pi]) \} $$ will be denoted by $L_{bc},$ or $L_{Per^+},
L_{Per^-}, L_{Dir}, L_{Neu}$ correspondingly. $Dom(L_{bc})$ is dense
in $L^2([0,\pi])$ and  $L_{bc}=L_{bc}(v)$ satisfies
 \be
 \label{L-L}
 (L_{bc}(v))^*=L_{bc}(\bar{v})\;\;\;\;\;bc=Per^\pm,\;Dir,\;Neu,
 \ee
where $(L_{bc}(v))^*$ is the adjoint operator and  $\bar{v}$ is the
conjugate of $v$, i.e., $\bar{v}(x)=\overline{v(x)}$. In the
classical case where $v\in L^2([0,\pi])$, (\ref{L-L}) is a well known
fact. In the case where $v\in H^{-1}_{per}(\mathbb{R})$ it is
explicitly stated and proved for $bc=Per^\pm,\;Dir$ in \cite{DM16},
see Theorem~6 and Theorem~13 there. Following the same argument as in
\cite{DM16} one can easily see that it holds for $bc=Neu$ as well.

If $v=0$ we write $L^0_{bc},$ (or simply $L^0$). The spectra and
eigenfunctions of $L^0_{bc}$ are as follows:

($\tilde{a}$) $ Sp (L^0_{Per^+}) = \{n^2, \; n = 0,2,4, \ldots \};$
its eigenspaces are  $\mathcal{E}^0_n = Span \{e^{\pm inx} \} $ for
$n>0 $ and $E^0_0 = \{ const\}, \; \; \dim \mathcal{E}^0_n = 2 $ for
$n>0, $ and $\dim \mathcal{E}^0_0 = 1. $

($\tilde{b}$) $ Sp (L^0_{Per^-}) = \{n^2, \; n = 1,3,5, \ldots \};$
its eigenspaces are  $\mathcal{E}^0_n = Span \{e^{\pm inx} \}, $ and
$ \dim \mathcal{E}^0_n = 2. $

($\tilde{c}$) $ Sp (L^0_{Dir}) = \{n^2, \; n \in \mathbb{N} \};$ each
eigenvalue $n^2 $ is simple; its eigenspaces are $\mathcal{S}^0_n =
Span \{ s_n (x)\}, $ where $s_n (x)$ is the corresponding normalized
eigenfunction $ s_n (x) = \sqrt{2} \sin nx.$

($\tilde{d}$) $ Sp (L^0_{Neu}) = \{n^2, \; n \in \{0\}\cup\mathbb{N}
\};$ each eigenvalue $n^2 $ is simple; its eigenspaces are
$\mathcal{C}^0_n = Span \{ c_n (x)\}, $ where $c_n (x)$ is the
corresponding normalized eigenfunction $ c_0 (x) = 1$, $c_n (x) =
\sqrt{2} \cos nx $ for $n>0.$

 The sets of indices $2\mathbb{Z}$,
 $2\mathbb{Z}+1$,  $\mathbb{N}$, and $\{0\}\cup\mathbb{N}$ will be denoted by
 $\Gamma_{Per^+}$, $\Gamma_{Per^-}$, $\Gamma_{Dir}$ and
 $\Gamma_{Neu}$, respectively. For each $bc$
 we consider the corresponding canonical orthonormal basis of $L^0_{bc}$, namely
 $\mathcal{B}_{Per^+}=\{e^{inx}\}_{n\in\Gamma_{Per^+}}$,
  $\mathcal{B}_{Per^-}=\{e^{inx}\}_{n\in\Gamma_{Per^-}}$,
 $\mathcal{B}_{Dir}=\{s_n(x)\}_{n\in\Gamma_{Dir}}$, $\mathcal{B}_{Neu}=\{c_n(x)\}_{n\in\Gamma_{Neu}}.$

 The matrix elements of an operator $A$ with respect to the basis $\mathcal{B}_{bc}$ will
 be denoted by $A^{bc}_{nm}$, where $n,m \in \Gamma_{bc}$. The norm of an
 operator $A$ from $L^a([0,\pi])$ to $L^b([0,\pi])$ for $a,b \in [1,\infty]$ will be denoted
 by $\|A\|_{a\rightarrow b}$. We may also write $\|A\|$ instead of $\|A\|_{2\rightarrow
 2}$.

Let $V$ denote the operator of multiplication by $v$, i.e., $
(Vf)(x)=v(x)f(x).$ Then $L_{bc}=L^0_{bc}+V$ and we may use the
perturbation formula (see \cite{DM16}, equation (5.13))
 \be
 \label{2-res}
  R_{\lambda}=R^0_{\lambda}+\sum_{m=1}^{\infty}K_{\lambda}(K_{\lambda}V
  K_{\lambda})^m K_{\lambda},
  \ee
  where $R_{\lambda}=(\lambda-L_{bc})^{-1}$, $R^0_{\lambda}=(\lambda-L^0_{bc})^{-1}$ and $K_{\lambda}$ is
  a square root of $R^0_{\lambda}$, i.e.,  $K_{\lambda}^2=
  R^0_{\lambda}$. Of course, (\ref{2-res}) makes sense only if the
  series on the right converges.

  Since the matrix representation of $R^0_{\lambda}$ is
  \be
    (R^0_{\lambda})^{bc}_{ij}=\frac{1}{\lambda-j^2}\delta_{ij},
    \quad i,j\in \Gamma_{bc}
    \ee
we can define a square root $K=K_{\lambda}$ of $R^0_{\lambda}$ by
choosing its matrix representation as
    \be
    (K_{\lambda})^{bc}_{ij}=\frac{1}{(\lambda-j^2)^{1/2}}\delta_{ij}, \quad i,j\in
    \Gamma_{bc},
    \ee
where $z^{1/2}$ is the principal square root. If
 \be
 \label{2-kvk}
 \|K_{\lambda}V K_{\lambda}\|_{2\rightarrow 2} < 1,
 \ee
then $R_{\lambda}$ exists. Assuming only $v\in
H_{per}^{-1}(\mathbb{R})$, Djakov and Mityagin showed (see
\cite{DM16}, Lemmas 19 and 20) that there exists $N>0$, $N\in
\Gamma_{bc}$ such that (\ref{2-kvk}) holds for $\lambda \in H^N
\backslash R_N$ and also for all $n>N$, $n\in \Gamma_{bc}$
(\ref{2-kvk}) holds for $\lambda \in H_n \backslash D_n$ if
$bc=Per^{\pm}$ and for $\lambda \in G_n \backslash D_n$ if $bc=Dir$
where
 \be
 H^N=\{\lambda \in \mathbb{C} : Re \; \lambda  \leq N^2+N\},
 \ee
 \be
 R_N=\{\lambda \in \mathbb{C} : -N\leq Re \;\lambda\leq N^2+N,\quad |Im \lambda|<N\},
 \ee
 \be
 H_n=\{\lambda \in \mathbb{C} : (n-1)^2\leq Re \;\lambda \leq (n+1)^2\},
 \ee
 \be
 G_n=\{\lambda \in \mathbb{C} : n^2-n\leq Re \;\lambda \leq n^2+n\},
 \ee
 \be
 D_n=\{ \lambda \in \mathbb{C} : |\lambda -n^2|< n/4\}.
 \ee
Therefore, the following localization of the spectra holds:
  \be
  \label{local}
   Sp(L_{bc}) \subset R_N \cup \bigcup_{n>N, n\in
\Gamma_{bc}}D_n, \;\;\; bc=Per^\pm,\;Dir.
 \ee
Moreover, using the method of continuous parametrization of the
potential $v$, they showed that spectrum is discrete and
$$
 \sharp(Sp(L_{Per^+})\cap R_N)=2N+1, \quad  \sharp(Sp(L_{Per^+})\cap D_n)=2,\quad n>N, n\in
 \Gamma_{Per^+},
$$$$
 \sharp(Sp(L_{Per^-})\cap R_N)=2N, \quad  \sharp(Sp(L_{Per^-})\cap D_n)=2,\quad n>N, n\in
 \Gamma_{Per^-} ,
$$$$
 \sharp(Sp(L_{Dir})\cap R_N)=N, \quad  \sharp(Sp(L_{Dir})\cap D_n)=1,\quad n>N, n\in
 \Gamma_{Dir}.
$$

For Neumann $bc$, using the same argument as in \cite{DM16} one can
similarly localize and count the spectrum $Sp(L_{Neu})$ after
showing that (\ref{2-kvk}) holds for $ \lambda \not \in R_N \cup
\left \{\bigcup_{n>N, n\in \Gamma_{Neu}}D_n \right \}.$ However in
the case where $v\in L^1([0,\pi])$ we  estimate $\|K_{\lambda}V
K_{\lambda}\|$ explicitly for all $bc$ since we  need this estimate
later.
 \begin{Proposition}
 \label{2-p-kvk} If  $v\in L^1([0,\pi])$, there
 exist $C=C(v)$ and $N>0$, $N\in\Gamma_{bc}$  such that
\be
 \|K_{\lambda}V K_{\lambda}\| \leq C \frac{\log N}{\sqrt{N}} \quad
 \text{for} \quad \lambda \in H^N \backslash R_N.
 \ee
 Moreover for all $n>N$, $n\in\Gamma_{bc}$,
 \be
 \label{KVK}
 \|K_{\lambda}V K_{\lambda}\| \leq C \frac{\log n}{n}
 \ee
 for $\lambda \in H_n \backslash D_n$ if
 $bc=Per^{\pm},$ and for $\lambda \in G_n \backslash D_n$ if $bc=Dir,
 Neu$.
 \end{Proposition}
\begin{proof} Since the matrix
representation of $K_{\lambda}V K_{\lambda}$ is
 \be
    (K_{\lambda}V
    K_{\lambda})^{bc}_{kj}=
    \sum_{k,j\in\Gamma_{bc}}
    \frac{V^{bc}_{kj}}{(\lambda-k^2)^{1/2}(\lambda-j^2)^{1/2}}
    \ee
we can estimate its Hilbert-Schmidt norm as
 \be
 \label{ineq}
 \|K_{\lambda}VK_{\lambda}\|_{HS}=
 \bigg{(}\sum_{k,j\in\Gamma_{bc}}\frac{|V^{bc}_{kj}|^2}
 {|\lambda-k^2||\lambda-j^2|}\bigg{)}^{1/2}\quad\quad\quad\quad
 \ee
 $$\quad\quad\quad\quad \quad\quad\quad\leq \sup_{k,j \in
 \Gamma_{bc}}|V^{bc}_{kj}|\bigg{(}\sum_{k\in\Gamma_{bc}}
 \frac{1}{|\lambda-k^2|}\sum_{j\in\Gamma_{bc}}\frac{1}{|\lambda-j^2|}\bigg{)}^{1/2}.$$
From (5.27) and (5.28) in \cite{DM16} we know that there exists an
$N>0$ such that for all $n>N$ there exists an absolute constant
$C_1$ such that
 \be
 \label{log}
 \sum_{k\in\Gamma_{bc}}\frac{1}{|\lambda-k^2|} \leq  C_1\frac{\log n}{n}
 \ee
for $\lambda \in H_n\setminus D_n$ if $bc=Per^{\pm}$ and $\lambda
\in G_n\setminus D_n$ if $bc=Dir, Neu$. On the other hand to
estimate $\sup_{k,j \in
 \Gamma_{bc}}|V^{bc}_{kj}|$ let us denote the Fourier coefficients of $v$
with respect to the bases $\mathcal{B}_{Per^+}$ and
$\mathcal{B}_{Neu}$ by $V_+(k)$ and $V_c(k)$, respectively, i.e.,
 \be
  V_+(k)=\frac{1}{\pi}\int_0^{\pi}v(x)e^{-ikx}dx, \quad k\in
  \Gamma_{Per^+},
  \ee
     \be
  V_c(k)=\frac{1}{\pi}\int_0^{\pi}v(x)\sqrt{2}\cos(kx)dx, \quad k\in \Gamma_{Neu}.
  \ee
  Then one can easily see that
  \be \label{r1} V^{Per^\pm}_{kj}=V_+(j-k),  \quad k,j \in
\Gamma_{Per^\pm}, \ee
  \be
  \label{r2}  V^{Dir}_{kj}=V_c(|j-k|)-V_c(j+k),  \quad
k,j \in \Gamma_{Dir},\ee
 \be
 \label{r3} V^{Neu}_{kj}=V_c(|j-k|)+V_c(j+k),
\quad k,j \in \Gamma_{Neu}.\ee
 Hence
  \be
  \label{in1}
  \sup_{k,j \in
 \Gamma_{Per^\pm}}|V^{Per^\pm}_{kj}| \leq \sup_{k \in \Gamma_{Per^+}}|V_+(k)| \leq
 \|v\|_1,
 \ee
 and
 \be
 \label{in2}
 \sup_{k,j \in
 \Gamma_{Dir}}|V^{Dir}_{kj}|,\;\; \sup_{k,j \in
 \Gamma_{Neu}}|V^{Neu}_{kj}| \leq 2\sup_{k \in \Gamma_{Neu}}|V_c(k)| \leq
 2\sqrt{2}\|v\|_1.
 \ee
Now, together with (\ref{log}), (\ref{in1}) and (\ref{in2}),
(\ref{ineq}) implies that
 \be
\|K_{\lambda}VK_{\lambda}\|_{HS} \leq 3\|v\|_p C_1\frac{\log n}{n}
 \ee
for $\lambda \in H_n\setminus D_n$ if $bc=Per^{\pm}$ and $\lambda
\in G_n\setminus D_n$ if $bc=Dir, Neu$.

On the other hand, if $\lambda=x+it \in H^N\backslash R_N, \;$ with
$n^2 -n \leq x < n^2 +n$ in the case when $x\geq 0,$ then one can
see that
\begin{equation}
\label{lamda-kkare} |\lambda-k^2| \geq
\begin{cases}(k^2+N)/\sqrt{2} & \text{if} \quad x \leq 0, \\
(|n^2-k^2|+2N)/2\sqrt{2} & \text{if}
\quad x>0, \quad k \neq \pm n, \\
N & \text{if} \quad x >0, \quad k = \pm n.
\end{cases}
\end{equation}
By Lemma 79 in \cite{DM15}, for large enough $N$ we also have the
inequality
\begin{equation}
\label{101.11a} \sum_k \frac{1}{|n^2 -k^2|+N} \leq C_2 \frac{\log
N}{\sqrt{N}},
\end{equation}
where $C_2$ is an absolute constant. Combining (\ref{ineq}),
(\ref{in1}), (\ref{in2}), (\ref{lamda-kkare}), and (\ref{101.11a})
we obtain for large enough $N$ that
 \be
\|K_{\lambda}VK_{\lambda}\|_{HS} \leq C_3 \|v\|_p \frac{\log
N}{\sqrt{N}} \quad \text{for} \quad \lambda \in H^N \backslash R_N,
 \ee
where $C_3$ is an absolute constant. We complete the proof noting
that Hilbert-Schmidt norm of an operator dominates its $L^2$ norm.
\end{proof}

\begin{Proposition} For any potential $v\in L^1([0,\pi])$,
the spectrum of the operator $L_{Neu}(v)$ is discrete. Moreover
there exists an integer $N$ such that
  \be
  \label{local2}
   Sp(L_{Neu}) \subset R_N \cup \bigcup_{n>N, n\in
\Gamma_{Neu}}D_n,
 \ee
and \be
 \sharp(Sp(L_{Neu})\cap R_N)=N+1, \quad
 \sharp(Sp(L_{Neu})\cap D_n)=1,\quad n>N, n\in
 \Gamma_{Neu}.
 \ee
 \end{Proposition}

\begin{proof} Apply the proof of Theorem 21 in \cite{DM16} but use
Proposition \ref{2-p-kvk} instead of Lemmas 19 and 20 in
\cite{DM16}.
\end{proof}

\section{Estimates of $\|P_n-P^0_n\|_{{}_{{}_{2\rightarrow\infty}}}$}

For $bc=Per^\pm, Dir$ or $Neu$, we consider the Cauchy-Riesz
projections
\begin{equation}
\label{PP0} P_n = \frac{1}{2 \pi i} \int_{C_n} R_\lambda d\lambda,
\quad P^0_n = \frac{1}{2 \pi i} \int_{C_n} R^0_\lambda d\lambda,
\end{equation}
where $C_n=\partial D_n$. In this section we estimate the norms
$\|P_n-P^0_n\|_{{}_{{}_{2\rightarrow\infty}}}$ and
$\|D(P_n-P^0_n)\|_{{}_{{}_{2\rightarrow\infty}}}$, where
$D=\frac{d}{dx}$. First we consider two technical lemmas.

\begin{Lemma}
\label{Alem} For every $p\in(1,2]$ there is a constant $T_p$ such
that
 \be \label{Ax1}
 \bigg{(}\sum_{k\in \Gamma_{bc}}^{\infty}\frac{k^{p\e}}
 {|\lambda-k^2|^p}\bigg{)}^{1/p} \leq
 \frac{T_p}{n^{(1-\e)}} \quad \text{for} \quad \lambda \in C_n,
 \ee
where $\e =0,1$ and $bc=Per^\pm, Dir, Neu$.
\end{Lemma}

\begin{proof} First note that for each $bc=Per^\pm, Dir, Neu$
 \be
\label{Ap-sum}
 \sum_{k\in \Gamma_{bc}}^{\infty}\frac{|k|^{p\e}}{|\lambda-k^2|^p}
 \leq 2\sum_{k=0}^{\infty}\frac{|k|^{p\e}}{|\lambda-k^2|^p}.
 \ee
If $\lambda \in C_n$ and $k \neq n$, then $|\lambda -n^2|=n/4,$ so
 \be
 \label{Ap-sum22}
  |\lambda-k^2| \geq |n^2-k^2|-\frac{n}{4} \geq \frac{1}{2}|n^2-k^2|
 \ee
because $|n^2-k^2| \geq 2n-1 \geq n.$ Hence
 \be
 \label{AI1}
 \sum_{{\tiny{\begin{array}{c}k=0 \\(k \neq n)
\end{array}}}}^{\infty}\frac{k^{p\e}}{|\lambda-k^2|^p}
\leq 2^p \sum_{{\tiny{\begin{array}{c}k=0 \\(k \neq n)
\end{array}}}}^{\infty}\frac{k^{p\e}}{|n-k|^p|n+k|^p},
 \ee
and the latter sum does not exceed
 \be
 \label{ApI2}
 \frac{2^p}{n^{(1-\e)p}}\sum_{{\tiny{\begin{array}{c}k=0
\\(k \neq n)
\end{array}}}}^{\infty}\frac{1}{|n-k|^p}\leq
\frac{2^{p+1}}{n^{(1-\e)p}}\sum_{\tiny{k=1}}^{\infty}\frac{1}{k^p}.
\ee
 Therefore, combining (\ref{Ap-sum}), (\ref{AI1}) and (\ref{ApI2}),
and taking into account that, for $\lambda\in C_n$, the $n^{th}$
summand of the right hand side of (\ref{Ap-sum}) is
$4^p/n^{(1-\e)p}$, we finally get
 $$ \label{Ax9}
 \sum_{k\in \Gamma_{bc}}^{\infty}\frac{k^{p\e}}{|\lambda-k^2|^p} \leq
 \frac{T^p_p}{n^{(1-\e)p}}
 $$
 with $T^p_p=2(4^p+
 2^{p+1}\sum_{\tiny{k=1}}^{\infty}\frac{1}{k^p})$. Hence (\ref{Ax1})
 holds.
\end{proof}

 \begin{Lemma}
\label{Av} Let $v\in L^p([0,\pi])$, $p\in (1,2]$, and let
$1/p+1/q=1$. Then, for $bc=Per^\pm,\; Dir,\; Neu$,  we have \be
\label{Ag}
\bigg{(}\sum_{k\in\Gamma_{bc}}|V^{bc}_{kj}|^q\bigg{)}^{1/q} \leq
5\|v\|_p \quad \forall j\in\Gamma_{bc}. \ee

\end{Lemma}
\begin{proof}
 From (\ref{r1}), (\ref{r2}) and (\ref{r3}) it follows
 \be \label{Ac1}
\left(\sum_{k\in\Gamma_{bc}}|V^{bc}_{kj}|^q \right)^{1/q}=
 \bigg{(}\sum_{k\in\Gamma_{bc}}|V_+(j-k)|^q\bigg{)}^{1/q}=\|V_+\|_{\ell^q}
 \ee
if $bc= Per^\pm,$ and
 \be  \label{Ac2}
 \bigg{(}\sum_{k\in\Gamma_{bc}}|V^{bc}_{kj}|^q\bigg{)}^{1/q} \leq
\bigg{(}\sum_{k\in\Gamma_{bc}}|V_c(|j-k|)|^q\bigg{)}^{1/q}+
\bigg{(}\sum_{k\in\Gamma_{bc}}|V_c(j+k)|^q\bigg{)}^{1/q}
 \ee
 $$ \leq 3\|V_c\|_{\ell^q} \quad \quad
 \quad\quad \quad \quad \quad \quad \quad \quad \quad$$
 if $bc= Dir$, $Neu$. On the other hand, by the Haussdorff-Young
(Hausdorff-Young Theorem, \cite[Theorem XII.2.3]{Zy90}),
 we have
 $\|V_+\|_{\ell^q} \leq \|v\|_p$ and $\|V_c\|_{\ell^q} \leq
 \sqrt{2}\|v\|_p$. These inequalities, together with (\ref{Ac1}) and
 (\ref{Ac2}), imply (\ref{Ag}).
\end{proof}

\begin{Proposition}
\label{Proposition} Let $D=\frac{d}{dx}$, $P_n$ and $P^0_n$ be
defined by (\ref{PP0}), and let $L=L_{bc}$ with $bc=Per^\pm, Dir,
Neu.$ If $v\in L^p([0,\pi])$, $p\in(1,2],$ then we have, for large
enough $n,$
 \be \label{P-P0}
\|P_n-P^0_n\|_{{}_{{}_{2\rightarrow\infty}}} \leq \frac{M}{n} \ee
and \be \label{DP-DP0}
\|D(P_n-P^0_n)\|_{{}_{{}_{2\rightarrow\infty}}} \leq M, \ee where
$M=M(v)$.
\end{Proposition}
\begin{proof} In view of (\ref{PP0}),
 \be
 \label{3-P}
\|P_n-P^0_n\|_{{}_{{}_{2\rightarrow\infty}}} \leq \frac{1}{2 \pi }
\int_{C_n}\|R_\lambda-R^0_\lambda\|_{{}_{{}_{2\rightarrow\infty}}}d|\lambda|\leq
n\sup_{\lambda \in C_n}
\|R_\lambda-R^0_\lambda\|_{{}_{{}_{2\rightarrow\infty}}}.
 \ee
In order to estimate
$\|D(P_n-P^0_n)\|_{{}_{{}_{2\rightarrow\infty}}}$, first we note that
 \be
 \label{3-D-P}
D(P_n-P^0_n) =
\frac{1}{2\pi}\int_{C_n}D(R_\lambda-R^0_\lambda)d\lambda.
 \ee
 Indeed, using integration by parts twice one can easily see that
 \be
 \label{3-D-D-2}
\bigg{\langle} D\int_{C_n}(R_\lambda-R^0_\lambda)f d\lambda ,
g\bigg{\rangle}=\bigg{\langle} \int_{C_n}D(R_\lambda-R^0_\lambda) f
d\lambda,g \bigg{\rangle}
 \ee
  for all $f\in L^2([0,\pi])$ and $g \in C^{\infty}_0([0,\pi])$. Since
  $C^{\infty}_0([0,\pi])$ is dense in
  $L^2([0,\pi]),$ (\ref{3-D-D-2}) implies (\ref{3-D-P}).
  Hence
   \be
   \label{3-D-P-2}
\|D(P_n-P^0_n)\|_{{}_{{}_{2\rightarrow\infty}}} \leq
\frac{1}{2\pi}\int_{C_n}
\|D(R_\lambda-R^0_\lambda)\|_{{}_{{}_{2\rightarrow\infty}}}d|\lambda|
 \ee
$$ \quad\quad\quad\quad\quad\quad\quad\leq n\sup_{\lambda \in C_n}
\|D(R_\lambda-R^0_\lambda)\|_{{}_{{}_{2\rightarrow\infty}}}.
$$
By (\ref{3-P}) and (\ref{3-D-P-2}) we see that in order to obtain
(\ref{P-P0}) and (\ref{DP-DP0}) it is enough to find good estimates
for $\|R_\lambda-R^0_\lambda\|_{{}_{{}_{2\rightarrow\infty}}}$ and
$\|D(R_\lambda-R^0_\lambda)\|_{{}_{{}_{2\rightarrow\infty}}}$ for
$\lambda\in C_n$. Next we will do that.

  Let $\e$ denote either $1$ or $0$, so
$D^\e=D$ for $\e=1$ and $D^\e=I$ for $\e=0$. Consider the operator
$D^{\e}(R_{\lambda}-R^0_{\lambda})$ for $\lambda\in C_n$. By
(\ref{2-res}) we can write
 \be
 D^\e (R_{\lambda}-R^0_{\lambda})=D^\e
 R^0_{\lambda}VR^0_{\lambda}+\sum_{m=1}^{\infty}D^\e
 R^0_{\lambda}VK_{\lambda}(K_{\lambda}VK_{\lambda})^m K_{\lambda}.
 \ee
Hence
 \be
 \label{3-de}
 \|D^\e (R_{\lambda}-R^0_{\lambda})\|_{{}_{{}_{2\rightarrow\infty}}} \leq \|D^\e
 R^0_{\lambda}VR^0_{\lambda}\|_{{}_{{}_{2\rightarrow\infty}}}
 \quad\quad\quad\quad\quad\quad\quad\quad\quad\ee
$$\quad\quad\quad\quad\quad\quad\quad\quad\quad+\sum_{m=1}^{\infty}\|D^\e
 R^0_{\lambda}VK_{\lambda}\|_{{}_{{}_{2\rightarrow\infty}}}
 \|K_{\lambda}VK_{\lambda}\|_{{}_{{}_{2\rightarrow2}}}^m
 \|K_{\lambda}\|_{{}_{{}_{2\rightarrow 2}}}.
$$
Since $\lambda\in C_n$, $\|K_{\lambda}VK_{\lambda}\|\leq 1/2$ for
sufficiently large $n$'s by Proposition \ref{2-p-kvk}. So
(\ref{3-de}) implies \be
 \label{3-de2}
 \|D^\e (R_{\lambda}-R^0_{\lambda})\|_{{}_{{}_{2\rightarrow\infty}}} \leq \|D^\e
 R^0_{\lambda}VR^0_{\lambda}\|_{{}_{{}_{2\rightarrow\infty}}}
+2\|D^\e
R^0_{\lambda}VK_{\lambda}\|_{{}_{{}_{2\rightarrow\infty}}}
\|K_{\lambda}VK_{\lambda}\|
 \|K_{\lambda}\|.
 \ee
It is easy to see that
 \be \label{3-K}
\|K_{\lambda}\|=\sup_{k \in
\Gamma_{bc}}{\frac{1}{|\lambda-k^2|^{1/2}}}=
\frac{1}{|\lambda-n^2|^{1/2}}=\frac{2}{\sqrt{n}},\;\;\;
 \lambda\in C_n.
 \ee
 Next we find estimates for $\|D^\e
 R^0_{\lambda}VR^0_{\lambda}\|_{{}_{{}_{2\rightarrow\infty}}}$ and $\|D^\e
 R^0_{\lambda}VK_{\lambda}\|_{{}_{{}_{2\rightarrow\infty}}}.$
 If we denote the Fourier coefficients of a function $f \in L^2([0,\pi])$
 with respect to the basis $\mathcal{B}_{bc}$ by $f^{bc}_j$, one can easily see that
 \be
 (D^\e R^0_{\lambda}VR^0_{\lambda}f)(x)=\sum_{k,j\in\Gamma_{bc}}\frac{(i
k)^\e
V^{bc}_{kj}f^{bc}_j}{(\lambda-k^2)(\lambda-j^2)}e^{ikx},\;\;\;bc=Per^\pm.
 \ee
 Noting also that $Ds_{k}=kc_k$ and $Dc_{k}=-ks_k$, we obtain for
 $bc=Dir,Neu$
  \be
 (D^\e
 R^0_{\lambda}VR^0_{\lambda}f)(x)=\sum_{k,j\in\Gamma_{Dir}}\frac{k^\e
 V^{Dir}_{kj}f^{Dir}_j}{(\lambda-k^2)(\lambda-j^2)}\big{(}(1-\e)s_k(x)+\e
 c_k(x)\big{)},
 \ee
 \be
 (D^\e R^0_{\lambda}VR^0_{\lambda}f)(x)=\sum_{k,j\in\Gamma_{Neu}}\frac{k^\e
 V^{Neu}_{kj}f^{Neu}_j}{(\lambda-k^2)(\lambda-j^2)}\big{(}(1-\e)c_k(x)-\e
 s_k(x)\big{)}.
 \ee
Since all basis functions are bounded above by $\sqrt{2}$ with
respect to supremum norm, the above expressions show that for each
$bc$ we have
 \be\label{2I3}\|D^\e R^0_{\lambda}VR^0_{\lambda}f\|_{\infty} \leq
 2\sum_{k,j\in\Gamma_{bc}}\frac{|k|^\e
 |V^{bc}_{kj}||f^{bc}_j|}{|\lambda-k^2||\lambda-j^2|}.
 \ee
We first apply the H\"older inequality about the index $k,$ and then
Lemma \ref{Alem} and Lemma \ref{Av} to get
 \be\label{2I4}\|D^\e R^0_{\lambda}VR^0_{\lambda}f\|_{\infty} \leq 2
 \sum_{j\in\Gamma_{bc}}\frac{1}{|\lambda-j^2|}|f^{bc}_{j}|
 \bigg{(}\sum_{k\in\Gamma_{bc}}\frac{|k|^{p\e}}{|\lambda-k^2|^p}\bigg{)}^{1/p}
 \bigg{(}\sum_{k\in\Gamma_{bc}}|V^{bc}_{kj}|^q\bigg{)}^{1/q}
 \ee
$$\leq 10 T_p \|v\|_p\frac{1}{n^{(1-\e)}}
 \sum_{j\in\Gamma_{bc}}\frac{1}{|\lambda-j^2|}|f^{bc}_{j}|.\quad\quad\quad$$
Applying the Cauchy inequality and Lemma \ref{Alem} one more time,
we finally get
 \be\label{2I5}\|D^\e R^0_{\lambda}VR^0_{\lambda}f\|_{\infty}
\leq 10 T_2 T_p \|v\|_p\frac{1}{n^{(2-\e)}}\|f\|_2.
 \ee
In a similar way it follows that
 \be
 \label{2I5a}\|D^\e R^0_{\lambda}VK_{\lambda}f\|_{\infty}
 \leq  2\sum_{k,j\in\Gamma_{bc}}\frac{|k|^\e
 |V^{bc}_{kj}||f^{bc}_j|}{|\lambda-k^2||\lambda-j^2|^{1/2}},
 \ee
so using the same argument as above but with (\ref{log}) at the last
step, we obtain
 \be
 \label{2I6}
  \|D^\e
R^0_{\lambda}VK_{\lambda}f\|_{\infty} \leq 10 C_1^{1/2}T_p
\|v\|_p\frac{(\log n)^{1/2}}{n^{(3/2-\e)}}\|f\|_2.
 \ee
Hence
 \be \label{2I7}
 \|D^\e R^0_{\lambda}VR^0_{\lambda}\|_{{}_{{}_{2\rightarrow\infty}}}
 \leq \frac{\tilde{M}}{n^{(2-\e)}},
 \ee
and
 \be \label{2I8}
 \|D^\e R^0_{\lambda}VK_{\lambda}\|_{{}_{{}_{2\rightarrow\infty}}}
 \leq \tilde{M}\frac{(\log n)^{1/2}}{n^{(3/2-\e)}},
 \ee
where $\tilde{M}=10 T_p \|v\|_p(T_2+ C_1^{1/2})$. Now we apply
(\ref{KVK}), (\ref{3-K}), (\ref{2I7}) and (\ref{2I8}) to the right
hand side of (\ref{3-de2}) and obtain
 \be
\label{3-de3} \|D^\e
(R_{\lambda}-R^0_{\lambda})\|_{{}_{{}_{2\rightarrow\infty}}} \leq
\frac{\tilde{M}}{n^{(2-\e)}}\bigg{(}1+C_1\frac{(\log
n)^{3/2}}{n}\bigg{)} \leq \frac{2\tilde{M}}{n^{(2-\e)}}
 \ee
for sufficiently large $n$'s. Finally, (\ref{3-P}), (\ref{3-D-P-2}),
and (\ref{3-de3}) imply (\ref{P-P0}) and (\ref{DP-DP0}), which
completes the proof.
 \end{proof}

\section{Proof of Theorem \ref{Asymp}}

In this section, we give a proof of Theorem \ref{Asymp}, i.e., we
show that the sequences $(|\gamma_n| +|\delta_n^{Neu}|)$  and
$(|\beta_n^- (z_n^*)|+|\beta_n^+ (z_n^*)|)$ are asymptotically
equivalent. The proof is based on the methods developed in
\cite{KaMi01,DM5,DM15}, but the technical details are different.

Let $L=L_{Per^{\pm}}$ and $L^0=L^0_{Per^{\pm}},$ and let $P_n$ and
$P_n^0$ be the corresponding projections defined by (\ref{PP0}).
Then $\mathcal{E}_n=Ran\,P_n  $  and $\mathcal{E}_n^0 =Ran\,P_n^0  $
are invariant subspaces of $L$ and $L^0,$ respectively. By Lemma 30
in \cite{DMH}, $\mathcal{E}_n$ has an orthonormal basis $\{f_n,
\varphi_n\}$ satisfying
 \be
  \label{f}
 Lf_n=\lambda^+_n f_n
 \ee
 \be
 \label{phi} L\varphi_n=\lambda^+_n \varphi_n -\gamma_n \varphi_n +\xi_n
f_n.
 \ee

\begin{Lemma}
\label{x+g-b+b} In the above notations, for large enough $n,$ \be
\label{9.1} \frac{1}{5}(|\beta^+_n(z_n^*)|+|\beta^-_n(z_n^*)| ) \leq
|\xi_n|+|\gamma_n| \leq 9 (|\beta^+_n(z_n^*)|+|\beta^-_n(z_n^*)| )
 \ee
\end{Lemma}

\begin{proof}  Indeed, combining (7.13) and (7.18) and (7.31)
in \cite{DMH} one can easily see that $|\xi_n| \leq
3(|\beta_n^+(z_n^*)|+|\beta_n^+(z_n^*)|)+4|\gamma_n|.$ This
inequality,  together with Lemma 20 in \cite{DMH}, implies that $
|\xi_n|+|\gamma_n| \leq 9 (|\beta_n^+(z_n^*)|+|\beta_n^+(z_n^*)|)$
for sufficiently large $n$'s. On the other hand by (7.31), (7.18),
and (7.14) in \cite{DMH} one gets  $
|\beta_n^+(z_n^*)|+|\beta_n^+(z_n^*)| \leq 5 (|\xi_n|+|\gamma_n|) $
for sufficiently large $n$'s.
\end{proof}

In the following, for simplicity, we suppress $n$ in all symbols
containing $n$. From now on, $P$  ($P^0$) denotes the Cauchy-Riesz
projection associated with $L$ ($L^0$) only. We denote the
projections associated with $L_{Neu}$  and $L^0_{Neu}$ by $P_{Neu}$
and $P^0_{Neu},$ respectively, and $\mathcal{C}=\mathcal{C}(v)$
denotes the one dimensional invariant subspace of
$L_{Neu}=L_{Neu}(v)$ corresponding to $P_{Neu}.$ We also set, for a
smooth function $u$ and a point $x_0 \in [0,\pi],$
$d_{x_0}(u)=\frac{du}{dx}(x_0).$

\begin{Lemma} Let $f, \varphi$ be an orthonormal basis in
$\mathcal{E}$ such that (\ref{f}) and (\ref{phi}) hold. Then there is
a unit vector $G=af+b\phii$ in $\mathcal{E}$ satisfying
 \be
 \label{bcH}
  d_0(G)=d_\pi(G)=0,
   \ee
and there is a unit vector $g\in\mathcal{C}$ satisfying
 \be \label{dere}
 \langle G,\bar{g}\rangle \delta^{Neu}=
  b \langle \varphi,\bar{g}\rangle \gamma - b\langle f,\bar{g}\rangle \xi
\ee such that $\langle G,\bar{g}\rangle \in \mathbb{R}$ and
 \be
 \label{4950}
  \langle G,\bar{g}\rangle \geq \frac{71}{72}
  \ee
 for sufficiently large $n$.
\end{Lemma}
(Remark. (\ref{bcH}) means that $G$ is in the domain of $L_{Neu}.$)

\begin{proof} If $ d_0(f)=0 $ then
$d_\pi(f)=0$ since $f$ is either a periodic or antiperiodic
eigenfunction. Hence we can set $G=f$. Otherwise we set
$\tilde{G}(x)=d_0(\varphi)f(x)-d_0(f)\varphi(x)$. Then
$G=\tilde{G}/\|\tilde{G}\|$ satisfies (\ref{bcH}) because the
functions $f$ and $\phii$ are simultaneously periodic or
antiperiodic.

By (\ref{bcH}), $G \in Dom(L)\cap Dom (L_{Neu}), $ so we have
$L_{Neu}G=LG.$ Hence it follows
 \be \label{111111}
L_{Neu}G=aLf+bL\phii=a\lambda^+f+b(\lambda^+\phii-\gamma\ \phii+\xi
f)\ee
 $$ \quad \quad \quad \quad \quad \quad =\lambda^+(a
f+b\phii)+b(\xi f-\gamma \phii)=\lambda^+G+b(\xi f-\gamma \phii).$$

Fix a unit vector $g\in \mathcal{C}$ so that
 \be
  \label{accreditation}
   \langle G,\bar{g}\rangle= |\langle G,\bar{g}\rangle|,
 \ee
Passing to conjugates in the equation $-g^{\prime \prime} +v(x) g=
\nu g $ one can see that
 \be
 \label{de}
 L_{Neu}(\bar{v})\bar{g}=\bar{\nu}\bar{g}.
 \ee
Taking inner product of both sides of (\ref{111111}) with $\bar{g}$
we get
 \be \label{cln1} \langle L_{Neu}G,\bar{g}\rangle =
\lambda^+\langle G,\bar{g}\rangle+b(\xi\langle
f,\bar{g}\rangle-\gamma \langle \phii,\bar{g}\rangle).
 \ee
On the other hand, by (\ref{L-L}) and (\ref{de}),
we have
   \be \label{cln2} \langle L_{Neu}(v)G,\bar{g}\rangle =
   \langle G,(L_{Neu}(v))^*\bar{g}\rangle=
\langle G,L_{Neu}(\bar{v})\bar{g}\rangle
  = \nu \langle G,\bar{g}\rangle.
 \ee
  Now (\ref{cln1}) and (\ref{cln2}) imply (\ref{dere}).

Let $G^0=P^0G$ and  $\bar{g}^0=P^0_{Neu}\bar{g};$ then
$\|G^0\|,\|\bar{g}^0\| \leq 1$ since $P^0$ and $P^0_{Neu}$ are
orthogonal projections and $G$ and $\bar{g}$ are unit vectors.

We have
 $$
 \label{3IKA}
 \langle G,\bar{g}\rangle
=
 \langle G^0,\bar{g}^0\rangle +
\langle G^0,\bar{g}-\bar{g}^0\rangle + \langle G-G^0,
\bar{g}\rangle,
$$
so by the triangle and Cauchy inequalities it follows that
 $$
 \label{bodorey}
 |\langle G,\bar{g}\rangle| \geq |\langle G^0,\bar{g}^0\rangle|
 -\|\bar{g}-\bar{g}^0\|
- \|G-G^0\|.
 $$
By Proposition \ref{Proposition} we have
 \be \label{G-G0}
\|G-G^0\|=\|(P-P^0)G\| \leq \|P-P^0\| \leq \frac{M}{n}
 \ee
 and similarly
  \be
  \label{g-g0}
  \|\bar{g}-\bar{g}^0\|=
  \|(P_{Neu}(\bar{v})-P^0_{Neu})\bar{g}\|
  \leq \|P_{Neu}(\bar{v})-P^0_{Neu}\| \leq
\frac{M}{n}.
 \ee
  Hence, it follows that
 \be \label{oyleboyle}
  |\langle G,\bar{g}\rangle| \geq |\langle G^0,\bar{g}^0\rangle |
-\frac{2M}{n}.
 \ee
Next we estimate  $|\langle G^0,\bar{g}^0\rangle |$ from below in
order to get a lower bound for $|\langle G,\bar{g}\rangle |$. Since
$\mathcal{C}^0$ is spanned by $c_n(x)= \sqrt{2}\cos nx,$ $\bar{g}^0$
is of the form
 \be \label{g0}
 \bar{g}^0 = e^{i\theta}\|\bar{g}^0\|\bigg{(}\frac{1}{\sqrt{2}}e^{inx}+
 \frac{1}{\sqrt{2}}e^{-inx}\bigg{)}
 \ee
 for some $\theta \in [0,2\pi)$. Now let $G^0_1$ and $G^0_2$
 be the coefficients of $G^0$ in the basis
 $\{e^{inx},e^{-inx}\}$, i.e.,
 \be
 \label{nelabel}
 G^0(x)=G^0_1e^{inx}+G^0_2e^{-inx}.
 \ee
 Clearly $d_0(G^0)=in\,(G^0_1-G^0_2)$. Since
 $d_0(G)=0$, by Proposition \ref{Proposition} we also have
 \be |d_0(G^0)|=|d_0(G)-d_0(G^0)| \leq\|(D
 P-DP^0)G\|_{\infty} \leq M.
 \ee
 Hence we obtain
 \be
 \label{ne}
 |G^0_1-G^0_2| \leq \frac{M}{n}
 \ee
 and
 \be \label{nelabelne} |G^0_2| \leq |G^0_1|+ |G^0_1-G^0_2| \leq
 |G^0_1|+\frac{M}{n}.
  \ee
 From (\ref{G-G0}) it follows that
  $$
 \sqrt{|G^0_1|^2+|G^0_2|^2}=\|G^0\| \geq \|G\|-\|G-G^0\| \geq
 1-\frac{M}{n},
  $$
so by (\ref{nelabelne}) we get
 \be
 \label{nelabelnelabel}
 |G^0_1| \geq \frac{1}{\sqrt{2}}-\frac{2M}{n}.
 \ee
On the other hand (\ref{g0}) and (\ref{nelabel})  imply
 \be
 \label{g0G0}
 |\langle G^0,\bar{g}^0 \rangle|=\frac{1}{\sqrt{2}}\|\bar{g}^0\||G^0_1+G^0_2|
 \geq \frac{1}{\sqrt{2}}\|\bar{g}^0\|\big{(}2|G^0_1|-|G^0_1-G^0_2|\big{)}.
 \ee
 Combining (\ref{ne}), (\ref{nelabelnelabel}), (\ref{g0G0})
 and taking into account that
 $$\|\bar{g}^0\| \geq \|\bar{g}\|-\|\bar{g}-\bar{g}^0\| \geq
1-\frac{M}{n}$$ due to (\ref{g-g0}), we obtain
 \be
 |\langle G^0,\bar{g}^0 \rangle| \geq 1-\frac{6M}{n}
 \ee
 which, together with (\ref{oyleboyle}) and (\ref{accreditation}), implies
 \be
 \langle G,\bar{g}\rangle \geq 1-\frac{8M}{n}.
 \ee
Hence, for a sufficiently large $n$, $\langle G,\bar{g}\rangle \geq
71/72$.
 \end{proof}

 \begin{Corollary}
 \label{corollary}
 For sufficiently large $n$, we have
 \be
 \label{corr2} |\gamma_n|+|\delta^{Neu}_n| \leq 19
 \big{(}|\beta^+_n(z_n^*)|+|\beta^-_n(z_n^*)|\big{)}.
 \ee
 \end{Corollary}
\begin{proof}
Using (\ref{dere}), (\ref{4950}) and noting also that the absolute
values of $b$ and all inner products in the right-hand side of
(\ref{dere}) do not exceed $1$ we get $|\delta^{Neu}| \leq
72/71\big{(}|\xi|+|\gamma|\big{)}.$ This inequality, together with
Lemma \ref{x+g-b+b}, implies (\ref{corr2}).
 \end{proof}

Corollary \ref{corollary} proves the second inequality in
(\ref{A10}). In order to complete the proof of Theorem~\ref{Asymp} it
remains to prove the first inequality in (\ref{A10}).

By Proposition 34 in \cite{DMH}, if
 \be \label{case1}
  \textit{Case 1}: \quad \quad \quad \frac{1}{4}|\beta^-(z^+)| \leq
 |\beta^+(z^+)| \leq 4 |\beta^-(z^+)|,\quad\quad\quad\quad\quad
 \ee
 then we have
 \be
 \label{case11}
 |\beta^+(z^*)|+|\beta^-(z^*)| \leq 2|\gamma|.\quad \quad
 \ee
 Next we consider the complementary cases
$$
 \textit{Case 2(a)}:\quad 4|\beta^+(z^+)| < |\beta^-(z^+)|\quad
 \text{or} \quad \textit{Case 2(b)}:\quad 4|\beta^-(z^+)| <
 |\beta^+(z^+)|.$$
\begin{Lemma}
\label{LLL} If {\em Case 2(a)} or {\em Case 2(b)} holds, then we
have, for sufficiently large $n,$
 \be \label{LLLL}
\frac{1}{4} \leq \frac{|d_0(f)|}{|d_0(\phii)|} \leq 4.
 \ee
 \end{Lemma}

\begin{proof} We consider only  \textit{Case 2(a)}, since the proof \textit{Case 2(b)} is
 similar. Let $f^0=P^0f$, $\phii^0=P^0\phii$ and let
$ f^0=f^0_1 e^{inx}+f^0_2 e^{-inx} $ and $ \phii^0=\phii^0_1
e^{inx}+\phii^0_2 e^{-inx}. $
 In \textit{Case 2(a)}, if $v \in L^2([0,\pi])$
 it was shown in the proof of Lemma 64 in \cite{DM15} that the following
 inequalities hold (inequalities (4.51), (4.52), (4.54), and (4.55) in \cite{DM15}):
 \be
 \label{6-f}
 |f^0_1| \geq \frac{2}{\sqrt{5}}-2\kappa_n,
 \quad|f^0_2| \leq \frac{1}{\sqrt{5}},\quad |\phii^0_1| \leq \frac{1}{\sqrt{5}}+\kappa_n,
 \quad|\phii^0_2|\geq
 \frac{2}{\sqrt{5}}-2\kappa_n,
 \ee
where $\kappa_n$ is a sequence converging to zero. These
inequalities were derived using Lemma 21 and Proposition 11 in
\cite{DM15} which still hold in the case where $v \in
H^{-1}([0,\pi])$, including our case $v\in L^p([0,\pi])$ (See Lemma
6 and Proposition 44 in \cite{DMH}). Hence we can safely use them.

Note that $d_0(f^0)=in(f^0_1-f^0_2)$. Using (\ref{6-f}) we get
 \be
 |d_0(f^0)| \geq n(|f^0_1|-|f^0_2|) \geq
 n\bigg{(}\frac{1}{\sqrt{5}}-2\kappa_n\bigg{)} \geq
 \frac{n}{\sqrt{6}}
 \ee
for sufficiently large $n$. On the other hand we have
 \be
 |d_0(f^0)| \leq n(|f^0_1|+|f^0_2|) \leq n\sqrt{2}\|f^0\| \leq
 \sqrt{2}n
 \ee
Following the same argument for $d_0(\phii^0)$, we have both
 \be
  \frac{n}{\sqrt{6}} \leq |d_0(f^0)| \leq \sqrt{2}n\quad\text{and}\quad
  \frac{n}{\sqrt{6}} \leq |d_0(\phii^0)| \leq \sqrt{2}n.
 \ee
On the other hand by Proposition \ref{Proposition} we have
 \be
 |d_0(f)-d_0(f^0)| \leq M \quad \text{and} \quad |d_0(\phii)-d_0(\phii^0)| \leq M.
 \ee
 Hence, for sufficiently large $n$'s, we get
 \be
 \frac{|d_0(f)|}{|d_0(\phii)|} \leq
 \frac{|d_0(f^0)|+M}{|d_0(\phii^0)|-M}\leq
 \frac{\sqrt{2}n+M}{n/\sqrt{6}-M}\leq 4
 \ee
 and
 \be
 \frac{|d_0(f)|}{|d_0(\phii)|} \geq
 \frac{|d_0(f^0)|-M}{|d_0(\phii^0)|+M}\geq
 \frac{n/\sqrt{6}-M}{\sqrt{2}n+M}\geq \frac{1}{4}.
 \ee
 \end{proof}

\begin{Proposition}
 \label{prop2}
For sufficiently large $n$, we have
 \be
 \label{corr} \big{(}|\beta^+_n(z_n^*)|+|\beta^-_n(z_n^*)|\big{)} \leq
 80 (|\gamma_n|+|\delta^{Neu}_n|)
 \ee
 \end{Proposition}

\begin{proof} In view of
 (\ref{case11}), it remains to prove (\ref{corr}) if
 \textit{Case 2(a)} or \textit{Case 2(b)} holds.

 Now (\ref{dere}) implies that
 \be \label{yokoyleyok}
 |b||\langle f,\bar{g}\rangle | |\xi| \leq |
 \delta^{Neu}|+|\gamma|.
 \ee
 Thus, in order to estimate $|\xi|$ from above by
 $|\delta^{Neu}|+|\gamma|$ we need to
 find a lower bound to $|b||\langle f,\bar{g}\rangle |$. We have
 \be
 \label{aydiaydi}
 |b||\langle f,\bar{g}\rangle |=|b|\big{|}\langle f,G \rangle+\langle f,\bar{g}-G
 \rangle\big{|}\geq |a||b|-\|\bar{g}-G\|
 \ee
since $\|f\|=1$, $|b| \leq 1$ and $\langle f,G \rangle= \bar{a}.$ In
view of (\ref{4950})
 \be
||\bar{g}-G||^2=||\bar{g}||^2+||G||^2-2 Re \,
\langle\bar{g},G\rangle=2-2\langle\bar{g},G\rangle \leq
\frac{1}{36},
 \ee
  hence
 \be \label{uwl} \|\bar{g}-G\| \leq\frac{1}{6}. \ee
On the other hand, by the construction of $G$ we know $|b/a|=
|d_0(f)/d_0(\phii)|,$ so Lemma \ref{LLL} implies that $1/4\leq |b/a|
\leq 4.$ Since $|a|^2+|b|^2=1$, a standard calculus argument shows
that
 \be \label{sar1}
 |a||b| \geq\frac{4}{17}.
 \ee
In view of (\ref{uwl}) and (\ref{sar1}), the right-hand side of
(\ref{aydiaydi}) is not less than $4/17-1/6 > 1/15,$ i.e.,
$|b||\langle f,\bar{g}\rangle |>1/15.$ Hence, by (\ref{yokoyleyok}),
it follows that
 \be \label{bsbsb}
 |\xi| \leq 15 \big{(} |\delta^{Neu}|+|\gamma| \big{)}.
 \ee
Now we complete the proof combining (\ref{bsbsb}) and Lemma
\ref{x+g-b+b}. \end{proof}

Corollary \ref{corollary} and Proposition \ref{prop2} show that
(\ref{A10}) holds, so Theorem~\ref{Asymp} is proved.

\section*{Acknowledgement}
The author wishes to express his gratitude to Prof. Plamen Djakov
who suggested the problem and offered invaluable assistance, support
and guidance.

\end{document}